\documentclass[psamsfonts]{amsart}

\usepackage{tikz}
\usetikzlibrary{matrix,arrows}
\usepackage{amssymb,amsfonts}
\usepackage[all,arc]{xy}
\usepackage{enumerate}
\usepackage{mathtools}
\usepackage{mathrsfs}
\usepackage{cleveref}

\newtheorem{thm}{Theorem}[section]

\newtheorem{prop}[thm]{Proposition}
\newtheorem{lem}[thm]{Lemma}

\theoremstyle{definition}
\newtheorem{defn}[thm]{Definition}

\theoremstyle{remark}

\crefname{thm}{theorem}{theorems}
\crefname{lemma}{lemma}{lemmas}

\makeatletter
\let\c@equation\c@thm
\makeatother
\numberwithin{equation}{section}

\title{A P-local Delooping Machine}

\author{Matthew Sartwell}

\begin{document}

\maketitle

\begin{abstract}
We show that for spaces $A$ that satisfy a certain smallness condition, there
is a Lawvere theory $T_A$ so that a space $X$ has the structure of a
$T_A$-algebra if and only if $X$ is weakly equivalent to a mapping space out of
$A$. In particular, spheres localized at a set of primes satisfy this condition.
\end{abstract}

\section{Introduction}
It is a classical result that the existence of certain algebraic structures on
a space can determine whether or not it is an $n$-fold loop space
(where here $n$ is possibly $\infty$). For instance, Beck \cite{Beck} showed
that a space is an $n$-fold loop space if and only if it is an algebra over the
monad $\Omega^n \Sigma^n$. In \cite{May} May found a much simpler description
using monads which come from operads. Such monads have the useful property that
they are \textit{finitary}: completely determined by their restriction to finite
sets. Finitary monads and their algebras correspond to (Lawvere) theories and
their algebras \cite[II.4]{BV}.  

\begin{defn}\label{detectable}
A \textbf{theory} $T$ is a based simplicial category with objects $t_0,
t_1, t_2, ...$ such that $t_i$ is the $i$-fold product of $t_1$. In particular,
$t_0$ is the basepoint. An \textbf{algebra} over $T$ is a based, product
preserving functor $X$ from $T$ into based simplicial sets. The
\textbf{underlying space} of $X$ is $X(t_1)$.
\end{defn}

So an algebra over a theory consists of a space $X(t_1)$ along with operations
$X(t_1)^n \rightarrow X(t_1)$ that satisfy certain relations which are
parameterized by the theory $T$. The above discussion is saying that loop spaces
are detected by theories, in the following sense.

\begin{defn}
A space $A$ is \textbf{detectable} if there is a theory $T$ with the property
that a space $X$ is weakly equivalent to the underlying space of a $T$-algebra
if and only if $X$ is weakly equivalent to $\mathbf{Map_*}(A, Y)$ for some $Y$. 
\end{defn}

One of the upshots of a space being detectable is that algebras over a theory
are closed under various operations. For example, let $F$ be a functor from
spaces to spaces which preserves weak equivalences and preserves products up to
weak equivalences. Then \cite[Cor 1.4]{Badzioch} says that if $X$ is weakly
equivalent to an algebra over a theory $T$, so is $F(X)$. So if $A$ is
detectable, applying such a functor $F$ to a mapping space of the form
$\mathbf{Map}_*(A, Y)$ gives a space which is weakly equivalent to
$\mathbf{Map}_*(A, Z)$ for some space $Z$.

In particular, if $A$ is detectable, then the localization $L_f\mathbf{Map}_*(A,
Y)$ with respect to any map $f$ is weakly equivalent to $\mathbf{Map}_*(A, Z)$
for some space $Z$. In \cite{mappingspaces}, it is shown that if $A$ is a
finite, pointed CW-complex with the property that its mapping spaces are closed
under localization in this way, then $A$ has the rational homotopy type of a
wedge of spheres which are all the same dimension.

Based on this result, it is not clear if \emph{any} spaces are detectable other
than wedges of $n$-spheres. On the contrary, we show that any space which
satisfies a certain smallness condition is detectable. In particular, this
includes spheres localized at a set of primes. Our main result is the following:

\begin{thm}\label{MainTheorem}
Let $S^n_P$ be the $P$-local sphere for $P$ a set of primes, and $n \geq 2$.
There is a theory $T_{S^n_P}$ with the property that a space $X$ is
weakly equivalent to $\mathbf{Map_*}(S^n_P, Y)$ for some $Y$ if and only if $X$
is weakly equivalent to the underlying space of an $T_{S^n_P}$-algebra.
\end{thm}

As in \cite{BCV}, we will actually prove the stronger statement that there is a
Quillen equivalence between the category of algebras over the theory $T_{S^n_P}$
and the right Bousfield localization of spaces with respect to $S^n_P$.\\

\textit{Notation}

\begin{itemize}
 \item We work in the category of pointed simplicial sets, which we
denote by $\mathbf{sSet}_*$.

\vspace{.3cm} 

\item For a set of primes $P$ and $n \geq 2$, the $P$-local sphere $S^n_P$ is
the singularization of the mapping telescope:

\begin{center}
 $\mathbf{Tel}(S^n \xrightarrow{l_1} S^n \xrightarrow{l_2} S^n \rightarrow...)$
\end{center}

\noindent
where $\{l_i\}$ are the maps whose degree are the positive integers relatively
prime to the primes in $P$.

\vspace{.3cm}

\item We freely use the language of model categories. For an introduction see
\cite{DS}.

\end{itemize}

\section*{Acknowledgements}
I wish to thank my advisor, Bernard Badzioch, for sparking my interest in
homotopy theory and algebraic structures, and for his patience and comments
through several versions of this paper. I also would like to thank Bruce
Corrigan-Salter and Alyson Bittner for many illuminating conversations.

\section{The Canonical Theory of a Space and the Homotopy Theory of Algebras}

The category of pointed simplicial sets is equipped with a pointed mapping
space given by:

\begin{center}
 $\mathbf{Map}_*(A, X) = \text{Hom}_{\mathbf{sSet}_*}(\Delta^\bullet_+ \wedge
A, X)$ 
\end{center}

\noindent
where the simplicial structure is encoded by the cosimplicial structure of
$\Delta^\bullet_+$, which is the pointed cosimplicial space sending each $[n]
\in \Delta$ to the standard $n$-simplex with a disjoint basepoint. For any
pointed simplicial set $A$, we define a theory $T_A$ by:

\begin{center}
 $\mathbf{Hom}_{T_A}(t_n, t_m) = \mathbf{Map}_*(\bigvee_m
A, \text{Sing}|\bigvee_n A|)$
\end{center}

\noindent
That this actually is a theory is probably easiest seen by the fact that the
adjunction between pointed simplicial sets and topological spaces induces
isomorphisms of simplicial mapping spaces:

\begin{center}
$\mathbf{Map}_*(\bigvee_m A, \text{Sing}|\bigvee_n A|) \cong
\mathbf{Map}_*(\bigvee_m |A|, \bigvee_n |A|)$ 
\end{center}

\noindent
Thus we may view our theory as the opposite of the full subcategory of the
(pointed simplicial) category of pointed topological spaces consisting of wedges
of copies of $|A|$. Since $\bigvee_n |A|$ is the $n$-fold coproduct of $|A|$, it
becomes the $n$-fold product in the opposite category.

\begin{defn}\label{Canonical}
Let $A$ be a space. The theory $T_A$ just described is called the
\textbf{Canonical Theory Associated to} $A$.
\end{defn}

In \cite{BCV}, it was shown that for $A = S^n$, the canonical theory detects
spheres in the sense of definition \ref{detectable}. Indeed, for any space $A$,
the mapping space $\mathbf{Map}_*(A, Y)$ is the underlying space of the
$T_A$ algebra:

\begin{center}
$\Omega^A(Y): T_A \rightarrow \mathbf{sSet}_*$ \hspace{2cm} $t_i \mapsto
\mathbf{Map_*}(\bigvee_i A, Y)$ 
\end{center}

\noindent
The content of \cite{BCV} was that, in the case of a sphere, every algebra over
the canonical theory $T_{S^n}$ has underlying space weakly equivalent to an
$n$-fold loop space.

We recall the homotopy theory of algebras over a theory $T$. Let
$\mathbf{Alg}^T$ denote the category of algebras over a theory $T$, where
the morphisms are natural transformations. Then taking the underlying space
gives a functor $U(X) = X(t_1)$ from the category of algebras over $T$ to
pointed simplicial sets. It has a left adjoint $F$, the free algebra functor,
and this adjunction is used to lift the Quillen model category of pointed
simplicial sets to the category of algebras.

\begin{thm}\cite{Reedy}\label{AlgMod}
There is a model category structure on $\mathbf{Alg}^{T_A}$, where a map $\phi:
X \rightarrow Y$ between algebras is a:

\begin{enumerate}
 \item Weak equivalence if $U(\phi)$ is a weak equivalence in $\mathbf{sSet}_*$
 \item Fibration if $U(\phi)$ is a fibration in $\mathbf{sSet}_*$
 \item Cofibration if $\phi$ has the left lifting property with respect to the
acyclic fibrations.
\end{enumerate}

\noindent
Moreover, with this model structure, the free/forgetful adjunction $F \dashv U$
is a Quillen adjunction.

\end{thm}

The fibrant algebras are the algebras whose underlying spaces are fibrant
simplicial sets. In the remainder of this section, we define a useful cofibrant
replacement of an algebra.

Let $X$ be an algebra over a theory $T$. Define the simplicial $T$-algebra
$FU_\bullet X$ to be $(FU)^{n+1}(X)$ in simplicial degree $n$. The face
and degeneracy maps are defined using the unit $\eta: 1 \rightarrow UF$ and
counit $\epsilon: FU \rightarrow 1$ of the adjunction. Specifically, we have

\begin{center}
 $d_i := (FU)^{k} X \xrightarrow{(FU)^i \epsilon (FU)^{k-i}} (FU)^{k-1} X $
 
 $s_i := (FU)^{k} X \xrightarrow{(FU)^i F\eta U(FU)^{k-i}} (FU)^{k+1} X$
\end{center}

\noindent
The counit $\epsilon: FU X \rightarrow X$ induces a map $\epsilon_*:
|FU_\bullet X| \rightarrow X$ of $T$ algebras.

\begin{defn}\label{Bar}
 The map $\epsilon_*: |FU_\bullet X| \rightarrow X$ is called \textbf{Bar
Resolution of} $X$.
\end{defn}

\begin{thm}\label{Res}
The bar resolution $\epsilon_*: |FU_\bullet X| \rightarrow X$ is a cofibrant
replacement in the category of algebras.
\end{thm}
\begin{proof}
The fact that $\epsilon_*$ is a weak equivalence follows from a standard extra
degeneracy argument, as in \cite[Proposition 9.8]{May}. To complete the proof,
we need to show that $FU_\bullet X$ is cofibrant. We will do this by showing
that $FU_\bullet X$ is a Reedy cofibrant algebra in the category of simplicial
algebras. Then geometric realization, being a left Quillen adjoint, sends Reedy
cofibrant simplicial algebras to cofibrant algebras.

Let $\Delta^{\text{op}}_0$ be the full subcategory of $\Delta^{\text{op}}$
consisting of the degeneracy maps and only positive face maps. Like
$\Delta^{\text{op}}$, this is a Reedy category. Consider the functor $UF_\bullet
X$ from $\Delta^{\text{op}}_0$ to $\mathbf{sSet}_\bullet$ which is $(UF)^{n}UX$
in simplicial degree $n$ and whose face maps and degeneracy maps are similar to
$FU_\bullet X$. All of the latching maps of $UF_\bullet X$ are cofibrations, and
applying $F$ we see that the restriction of $FU_\bullet X$ to
$\Delta^{\text{op}}_+$ is Reedy cofibrant. By \cite[Proposition 3.17]{Lifting},
the simplicial algebra $FU_\bullet X$ is Reedy cofibrant as well.
\end{proof}

\section{Cellular Spaces and Homotopy Projectivity}

The mapping space algebra $\Omega^A$ in the previous section defines a
functor from $\mathbf{sSet}_*$ to the category of algebras $\mathbf{Alg}^{T_A}$
whose morphisms are natural transformations. This functor has a left
adjoint, which we call $B^A$, and this adjunction passes to the level of
homotopy categories.

\begin{thm}
Let $A$ be a space, and let $T_A$ be the canonical theory from definition
\ref{Canonical}. Then the adjunction

\begin{center}
$\mathbf{sSet}_*: \xleftrightharpoons[\Omega^A]{B^A} \mathbf{Alg}^{T_A}$
\end{center}

\noindent
is Quillen.
\end{thm}
\begin{proof}
It suffices to show that the right adjoint preserves fibrations between fibrant
objects and acyclic fibrations. If $f: X \rightarrow Y$ is a fibration between
fibrant objects, or an acylic fibration, then so is the induced map

\begin{center}
$\mathbf{Map}_*(A, X) \rightarrow \mathbf{Map}_*(A, Y)$
\end{center}

\noindent
In other words $U(f)$ is a fibration or acyclic fibration, and by definition of
the model structure on algebras, so is $\Omega(f)$.
\end{proof}

We will see that for spaces $A$ which satisfy a smallness condition (see
definition \ref{HomotopyProjective}) we can define a model category structure on
$\mathbf{sSet}_*$ so that this adjunction becomes a Quillen equivalence. When
this is possible, the right adjoint $\Omega^A$ has to reflect weak equivalences
between fibrant objects. The following model category structure, called the
right Bousfield localization of $\mathbf{sSet}_*$ with respect to $A$, can be
thought of as the most efficient model category which makes this happen. We
will denote it $R^A\mathbf{sSet}_*$

\begin{thm}\cite[Theorem 5.1.1]{Hirshhorn}
Let $A$ be a pointed simplicial set. There is a model category structure on
$\mathbf{sSet}_*$, where a map $f: X \rightarrow Y$ is a:

\begin{enumerate}
 \item Weak equivalence if the induced map $\mathbf{Map_*}(A,
|\text{Sing}(X)|) \rightarrow \mathbf{Map_*}(A, |\text{Sing}(Y)|)$ is a weak
equivalence.

\item Fibration if it is a Kan fibration.

\item Cofibration if it has the left lifting property with respect to all
acyclic fibrations.
\end{enumerate}
\end{thm}
\begin{proof}
This follows from a general theorem on existence of right Bousfield
localizations \cite[Theorem 5.1.1]{Hirshhorn}, since the model category of
pointed simplicial sets is proper and cellular
\cite[Proposition 5.1.8]{Hirshhorn}
\end{proof}

By \cite[5.1.5]{Hirshhorn}, the cofibrant spaces in $R^A\mathbf{sSet}_*$ are the
$A$-cellular spaces, which we define now.

\begin{defn}
A nonempty class of simplicial sets is said to be \textbf{closed} if it is
closed under homotopy colimits and weak equivalences. The smallest closed class
which contains a simplicial set $A$ is called the class of \textbf{$A$-cellular
spaces}, denoted \textbf{Cell} $A$.
\end{defn}

The $A$-cellular spaces are roughly the spaces built from $A$. For example, the
$S^n$-cellular spaces are the $(n-1)$-connected simplicial sets. For us, the
most important example is the following.

\begin{lem}
 Let $n \geq 2$, and $P$ be a set of primes. Then $\mathbf{Cell}(S^n_P)$ is the
class of $(n-1)$-connected $P$-local cell-complexes.
\end{lem}
\begin{proof}
By definition, $\mathbf{Cell}(S^n_P)$ is the smallest class of simplicial sets
containing $S^n_P$ which is closed under weak equivalences and homotopy
colimits. The constructions in \cite{sullivan} show that it is possible to
build any $P$-local space by homotopy colimits. It is shown in
\cite[5.3.7]{Hirshhorn} that it is enough to consider filtered colimits and
homotopy pushouts. Since the property of being $(n-1)$-connected and $P$-local
is determined by reduced homology, any homotopy pushout or filtered colimit of
$(n-1)$-connected $P$-local spaces is still $(n-1)$-connected and $P$-local.
\end{proof}

The final ingredient in our recognition principle is that $A$ needs to satisfy
a smallness condition.

\begin{defn}\label{HomotopyProjective}
Let $C$ be a closed class. A pointed simplicial set $A$ is
called \textbf{homotopy projective relative to} $C$ if $\mathbf{Map_*}(A, -)$
commutes with homotopy colimits of filtered or simplicial diagrams which take
their value in $C$. A space $A$ is called \textbf{homotopy self-projective} if
it is homotopy projective relative to \textbf{Cell} $A$.
\end{defn}

\begin{prop}
The spheres $S^n$ are homotopy self-projective.
\end{prop}
\begin{proof}
$\mathbf{Map_*}(S^n, -)$ commutes with homotopy filtered colimits because $S^n$
is a finite simplicial set. The class \textbf{Cell} $S^n$ is exactly the class
of $(n-1)$-connected spaces, so it follows from the Bousfield-Friedlander
theorem \cite[Theorem B.4]{BF} that $\mathbf{Map_*}(S^n, -)$ commutes with
homotopy colimits of $(n-1)$-connected simplicial spaces.
\end{proof}

\begin{prop} \label{hoproj}
For $n \geq 2$ and $P$ a set of primes, the $P$-local sphere $S^n_P$ is homotopy
self-projective.
\end{prop}
\begin{proof}
The map:

\begin{center}
 $\mathbf{Map_*}(S^n_P, K) \rightarrow \mathbf{Map_*}(S^n, K)$
\end{center}

\noindent
induced from the localization map $L_P: S^n \rightarrow S^n_P$ is a
weak equivalence for $P$-local spaces $K$. Let $\widetilde{K}: D \rightarrow
\mathbf{sSet}_*$ be either a filtered diagram or else a simplicial diagram which
takes its values in \textbf{Cell} $S^n_P$. We have the following commutative
diagram:

\begin{center}
\begin{tikzpicture}
  \matrix(m)[matrix of math nodes,row sep=2em,column sep=4em,minimum width=2em,
text height = 1.5 ex, text depth = 0.25]
  {\underset{D}{\text{hocolim}}\, \mathbf{Map_*}(S^n_P, \widetilde{K})  &
\mathbf{Map_*}(S^n_P, \underset{D}{\text{hocolim}}\, \widetilde{K}) \\
\underset{D}{\text{hocolim}}\, \mathbf{Map_*}(S^n, \widetilde{K}) &
\mathbf{Map_*}(S^n, \underset{D}{\text{hocolim}}\, \widetilde{K}) \\ };
  \path[-stealth]
    (m-1-1) edge node [above] {$\nu_P$} (m-1-2)
    (m-2-1) edge node [below] {$\nu$} (m-2-2)
    (m-1-1) edge node [left] {$L_P$} (m-2-1)
    (m-1-2) edge node [right] {$L_P$} (m-2-2);
\end{tikzpicture}
\end{center} 

The left and right maps are weak equivalences because $\widetilde K$ takes its
value in $P$-local spaces. The bottom map is a weak equivalence because $S^n$ is
homotopy self-projective and $\widetilde K$ takes its values in
$(n-1)$-connected spaces. Hence the top map is also a weak equivalence.

\end{proof}

These are actually the only homotopy self-projective spaces we know of, and we
think it is likely that wedges of $n$-spheres are the only finite spaces which
are homotopy self-projective. On the other hand, we can say that the
property of being homotopy self-projective is stable.

\begin{prop}
  If $A$ is connected and homotopy self-projective, so is $\Sigma A$.
\end{prop}
\begin{proof}
A space $X$ is in $\mathbf{Cell}(\Sigma A)$ if and only if $\Omega X$ is in
$\mathbf{Cell}(A)$ \cite{Farjoun}. So if $\widetilde K: D \rightarrow
\mathbf{S}_*$ is some filtered or simplicial diagram which takes its values in
$\mathbf{Cell}(\Sigma A)$, then the diagram $\Omega \widetilde K$ obtained by
postcomposition takes its values in $\mathbf{Cell}(A)$. Since $A$ is homotopy
self-projective

\begin{center}
$\underset{D}{\text{hocolim}}\,\mathbf{Map_*}(\Sigma A, \widetilde K) \cong
\underset{D}{\text{hocolim}}\, \mathbf{Map_*}(A, \Omega \widetilde K)
\rightarrow \mathbf{Map_*}(A, \underset{D}{\text{hocolim}}\, \Omega \widetilde
K)$ 
\end{center}

\noindent
is a weak equivalence. Since $\widetilde K$ takes its values
in connected spaces, and since $S^1$ is homotopy self-projective there
is a weak equivalence

\begin{center}
 $\mathbf{Map_*}(A, \underset{D}{\text{hocolim}} \, \Omega \widetilde K)
\rightarrow \mathbf{Map_*}(A, \Omega(\underset{D}{\text{hocolim}} \, \widetilde
K)) \cong \mathbf{Map_*}(\Sigma A, \underset{D}{\text{hocolim}}\, \widetilde K)$
\end{center}

\end{proof}

\section{Proof of the Main Theorem} \label{proof}

In this section we prove \cref{MainTheorem}. It will follow easily from
the following

\begin{thm} \label{Q}
Let $A$ be a homotopy self-projective space. Then the Quillen adjunction

\begin{center}
 $R^A\mathbf{sSet}_* \xleftrightharpoons[\Omega^A] {B^A}
\mathbf{Alg}^{T_A}$
\end{center}

\noindent
is a Quillen equivalence.
\end{thm}

\begin{proof}
We first need to show that this is still a Quillen adjunction after passing
from $\mathbf{sSet}_*$ to $R^A\mathbf{sSet}_*$. Note that the identity functor:

\begin{center}
 $Id: \mathbf{sSet}_* \rightarrow R^A\mathbf{sSet}_*$
\end{center}

\noindent
is a \emph{right} Quillen adjoint, so that this is not automatic. However
$R^A\mathbf{sSet}_*$ is a simplicial model category with the same simplicial
structure as $\mathbf{sSet}_*$ \cite[Theorem 5.1.2]{Hirshhorn}. So $\Omega^A$
preserves fibrations between fibrant objects and acyclic fibrations. 

Next, since $\Omega^A$ reflects weak equivalences between fibrant objects, it is
enough to show that the unit map $\eta: X \rightarrow \Omega^A(\text{Sing}|B^A
X|)$ is a weak equivalence for all cofibrant $X$. We split the proof into
cases: 

\textbf{1.} Consider first the case when $X = F(n_+)$, where $n = \{0, 1, ...,
n\}$ is the set based at $0$. For any space $K$, the composition $U \Omega^A(K)$
is the mapping space $\mathbf{Map}_*(A, X)$, so that in the other direction
$B^AF(K)$ is $A \wedge K$. Thus, applying the forgetful functor to the map $X
\rightarrow \Omega^A(\text{Sing}|B^A X|)$ is isomorphism:

\begin{center}
 $F(n_+) \rightarrow \mathbf{Map}_*(A, \text{Sing}|B^A F(n_+)|) \rightarrow
\mathbf{Map}_*(A, \text{Sing}|\bigvee_n A|)$
\end{center}

\textbf{2.} Next, let $S$ be a set, and let $X = F(S)$. Then we can write $X$
as the filtered homotopy colimit of over the finite subsets of $S$:

\begin{center}
 $X = \text{hocolim}_{L \subseteq S}\, F(L)$
\end{center}

We will write the unit as a composition. First there is the map
\begin{center}
 $\text{hocolim}_{L \subseteq S}\, F(L) \xrightarrow{\eta_\bullet}
\text{hocolim}_{L \subseteq S}\, \Omega^A (\text{Sing}|B^A F(L)|)$
\end{center}

\noindent
obtained by applying the unit degreewise. This map is a weak equivalence since
by part 1 it is an objectwise weak equivalence. We compose this with the map

\begin{center}
$\text{hocolim}_{L \subseteq S}\, \Omega^A (\text{Sing}|B^A F(L)|)
\xrightarrow{\nu} \Omega^A(\text{hocolim}_{L \subseteq S}\, \text{Sing}|B^A
F(L)|)$
\end{center}

\noindent
which is a weak equivalence because $A$ is homotopy projective. The final map
in the composition is

\begin{center}
$\Omega^A(\text{hocolim}_{L \subseteq S}\, \text{Sing}|B^A
F(L)|) \xrightarrow{i} \Omega^A( \text{Sing}|B^A \text{hocolim}_{L \subseteq
S}\, F(L)|) $
\end{center}

\noindent
This is a weak equivalence because $\text{hocolim}_{L \subseteq S}\,
\text{Sing}|B^A F(L)| \rightarrow \text{Sing}|B^A \text{hocolim}_{L \subseteq
S}\, F(L)|$ is a weak equivalence, by \cite[Proposition 18.9.12]{Hirshhorn}.

\textbf{3.} Next suppose that $X = \text{hocolim}_{\Delta^{\text{op}}}
K_\bullet$, where each $K_n$ is a free algebra $F(S)$. Then a similar argument
as in step 2 shows that the unit is a weak equivalence in this case.

\textbf{4.} Finally, let $X$ be a cofibrant algebra. Then the bar
resolution $\epsilon_*: \text{hocolim}_{\Delta^{\text{op}}}\, FU_\bullet(X)
\rightarrow X$ fits into the commutative diagram:

\begin{center}
\begin{tikzpicture}
  \matrix(m)[matrix of math nodes,row sep=4em,column sep=5em,minimum width=4em,
text height = 2.5 ex, text depth = 1]
  {\text{hocolim}_{\Delta^{\text{op}}}\, FU_\bullet(X) & \Omega^A
\text{Sing}|B^A \text{hocolim}_{\Delta^{\text{op}}}FU_\bullet(X)|    \\
   X & \Omega^A \text{Sing} |B^A X|   \\ };
  \path[-stealth]
    (m-1-1) edge node [above] {$\eta_\bullet$}  (m-1-2)
    (m-2-1) edge node [below] {$\eta$}  (m-2-2)
    (m-1-1) edge node [left]  {$\epsilon_*$} (m-2-1)
    (m-1-2) edge node [right]  {$\Omega^A\text{Sing}B^A \epsilon_*$} (m-2-2);
\end{tikzpicture}
\end{center} 

\noindent
The bar resolution has the form from step 3, so that $\eta_\bullet$ is a
weak equivalence. By theorem \label{Res}, the map $\epsilon_*$ is a weak
equivalence. Finally, $\Omega^A$ preserves weak equivalences between fibrant
objects, so that $\eta$ is a weak equivalence as well.
\end{proof}

\begin{proof}[Proof of Theorem \ref{MainTheorem}]
 We showed in section 2 that if a space $X$ is weakly equivalent to
$\mathbf{Map}_*(A, Y)$ for some $Y$, then $X$ is weakly equivalent to the
underlying space of the algebra $\Omega^A(Y)$. Conversely, suppose a space $X$
is weakly equivalent to the underlying space of an algebra $M$, then $X$ is also
weakly equivalent to the underlying space of a cofibrant replacement $M^c$ of
$M$. By \ref{Q}, there is a weak equivalence of algebras $M^c \rightarrow
\Omega^A (Sing|B^A M^c|)$. By the definition of weak equivalences between
algebras, it follows that $X$ is weakly equivalent to $\mathbf{Map}_*(A,
Sing|B^A M^c|)$.
\end{proof}

\end{document}